\documentclass[11pt,reqno]{amsart}

\usepackage{graphicx,color}
\usepackage{amssymb,amscd,latexsym, mathrsfs, epsfig}
\usepackage[all]{xy}

\newcommand\PP{\mathbb P}
\newcommand\C{\mathbb C}

\newcommand\Z{\mathbb Z}

\newcommand{\OO}{\mathcal{O}}
\newcommand{\X}{\mathcal{X}}
\newcommand{\Y}{\mathcal{Y}}

\newcommand{\LL}{\mathcal{L}}
\renewcommand{\Y}{\mathcal{Y}}

\newcommand{\calP}{\mathcal{P}}

\newcommand\Aut{\operatorname{Aut}}
\newcommand\df{\operatorname{DF}}
\newcommand\rees{\operatorname{Rees}}
\newcommand\proj{\operatorname{Proj}}
\newcommand\init{\operatorname{in}}
\newcommand\gr{\operatorname{gr}}
\newcommand\red{\operatorname{red}}

\newcommand\HM{\operatorname{HM}}

\newcommand\Chow{\operatorname{Chow}}

\newcommand\bra{\langle}
\newcommand\ket{\rangle}

\makeatletter \@addtoreset{equation}{section} \makeatother

\newtheorem{thm}{Theorem}
\newtheorem{prop}[thm]{Proposition}
\newtheorem{lem}[thm]{Lemma}
\newtheorem{cor}[thm]{Corollary}
\newtheorem{definition}[thm]{Definition}
\newtheorem{conj}[thm]{Conjecture}
\newenvironment{rmk}{\noindent\textbf{Remark}.}{}%{\\}
\newenvironment{exm}{\noindent\textbf{Example}.}{\\}

\title[]{Torus equivariant K-stability}
\author{Giulio Codogni}
\address{Giulio Codogni, EPFL, SB MATHGEOM CAG, MA B3 444 (B\^{a}timent A) Station 8, CH-105 Lausanne, Switzerland}
\email{giulio.codogni@epfl.ch}

\author{Jacopo Stoppa}
\address{SISSA, via Bonomea, 265 - 34136 Trieste ITALY}
\email{jstoppa@sissa.it}

\date{}

\begin{document}
\begin{abstract} It is conjectured that to test the K-polystability of a polarised variety it is enough to consider test-configurations which are equivariant with respect to a torus in the automorphism group. We prove partial results towards this conjecture. We also show that it would give a new proof of the K-polystability of constant scalar curvature polarised manifolds. 
\end{abstract}
\maketitle

\section{Introduction}

The Yau-Tian-Donaldson conjecture for Fano manifolds \cite{yau, tian, donaldson} % solved by Chen, Donaldson and Sun \cite{cds1, cds2, cds3, cds4},
 predicts that a smooth Fano $M$ admits a K\"ahler-Einstein metric if and only if it is K-polystable, a purely algebro-geometric condition expressed through the positivity of a certain limit of GIT weights (the Donaldson-Futaki weight or invariant). There are by now several proofs, in different degrees of generality (i.e. allowing $M$ to have mild singularities, a boundary in the MMP sense, and/or slightly modifying the notion of K-stability), using different methods.

For an arbitrary polarised manifold $(X,L)$ the most natural generalisation of a K\"ahler-Einstein metric is a constant scalar curvature K\"{a}hler (cscK) metric representing the first Chern class of $L$.  If such a metric exists, $(X,L)$ is called a cscK manifold. 

A K\"ahler-Einstein metric, or more generally a cscK metric, if it exists, can always be taken invariant under the action of a compact group of  automorphisms of $M$. From the GIT point of view, when the point whose stability we would like to investigate has a non-trivial reductive stabiliser $H$, the Hilbert-Mumford Criterion can be strengthened: it is enough to consider one-parameter subgroups which commute with $H$ \cite{Kempf}. These facts suggest the following folklore conjecture (all the notions required in the rest of this introduction will be recalled in Section \ref{backSec}.)
\begin{conj}\label{conj}
Let $(X,L)$ be a polarised variety and let $G$ be a reductive subgroup of $\Aut(X,L)$. Then $(X,L)$ is K-polystable if and only if for every $G$-equivariant test-configuration the Donaldson-Futaki invariant is greater than or equal to zero, with equality if and only if the normalisation of the test-configuration is a product.
\end{conj}

An analytic proof in the case of Fano manifolds is given in \cite{gaborCont}, relying on an alternative approach to the Yau-Tian-Donaldson conjecture.  An algebro-geometric proof in the Fano case and when $G$ is a torus is given in \cite{Chenyang_Equivariant}.%; the first part of their argument is quite similar to ours, then they manage to circumvent the problem of finite generation thanks to the graded ideals view-point.

Recall that a cscK manifold has reductive automorphism group, so K-polystable varieties are expected to have a reductive automorphism group as well; this problem is studied in \cite{giulio}. Because of this it is natural to formulate Conjecture \ref{conj} just for reductive subgroups of $\Aut(X,L)$.

There is a general expectation that for the existence of a cscK metric one actually needs some enhancement of the original notion of K-stability. Quite a few different notions have been proposed. In this paper we focus on the generalisation of K-stability based on (possibly non-finitely generated) filtrations of the coordinate ring of $(X,L)$ (see Definition \ref{def:cappuccio}). This notion has been proposed by G. Sz\'ekelyhidi in \cite{gabor}, building on the work of D. Witt Nystr\"{o}m \cite{nystrom}; in \cite{ICM}, it is called $\hat{K}$-stability. In \cite{gabor}, it is shown that, given a cscK manifold $(X,L)$, if the connected component of the identity of $\Aut(X,L)$ is equal to $\C^*$, then $(X,L)$ is $\hat{K}$-stable. Importantly for us \cite{gabor} also discusses a variant of $\hat{K}$-stability which replaces the Donaldson-Futaki invariant of a filtration with the asymptotic Chow weight $\Chow_{\infty}$, and proves that the $\hat{K}$-stability result remains true for this variant (the two notions coincide when dealing with classical test-configurations, corresponding to finitely generated filtrations).

Our main result is a step towards a proof of Conjecture \ref{conj} in the general case, or possibly of a variant of Conjecture \ref{conj} in the $\hat{K}$-stability setup. 
 
\begin{thm}\label{mainThm} 
Let $(X,L)$ be a polarised variety. Fix a complex torus $T \subset \Aut(X, L)$ and let $(\X,\LL)$ be a test-configuration with Donaldson-Futaki invariant $\df(\X, \LL)$. Then we can associate to $(\X,\LL)$ a $T$-equivariant filtration $\chi$ of the coordinate ring of $(X, L)$ whose asymptotic Chow weight satisfies $\Chow_{\infty}(\chi) \leq \df(\X, \LL)$. If moreover $\chi$ is finitely generated, then it corresponds to a $T$-equivariant test-configuration which is a flat one-parameter limit of $(\X,\LL)$, and in particular has the same Donaldson-Futaki invariant and $L^2$ norm.
\end{thm}
Theorem \ref{mainThm} follows at once from Lemma \ref{familyLem}, Lemma \ref{familyLem2} and Theorem \ref{main}, proved in Section \ref{SPEC}. Theorem \ref{main} shows that given a generalised test-configuration in the sense of \cite{gabor}, corresponding to a possibly non-finitely generated filtration $\chi$, we can specialise it to a $T$-invariant filtration $\bar{\chi}$ with $\Chow_{\infty}(\bar{\chi})\leq\Chow_{\infty}(\chi)$. In the Appendix we show that non-finitely generated filtrations can actually arise in Theorem \ref{mainThm}. 

In Section \ref{CSCK} we show that Conjecture \ref{conj} combined with ideas from \cite{jacopo,gaborJ} naturally leads to a  proof that cscK manifolds are K-polystable. K-polystability of cscK manifolds is proved in \cite{berman2} using completely different methods. 

\noindent\textbf{Notation.} In this paper a polarised variety $(X,L)$ is a complex projective variety $X$ endowed with a very ample and projectively normal line bundle $L$. For the purposes of this paper one may always replace $L$ with a positive tensor power, so these assumptions are not restrictive.

\noindent\textbf{Acknowledgements.} The second author learned about the equivariance question studied in this paper from S. K. Donaldson and discussed the problem and its implications with G. Sz\'ekelyhidi on several occasions. The present work is entirely motivated by those conversations.

We are also very grateful to R. Dervan, A. Ghigi, Y. Odaka, J. Ross, R. Svaldi, R. Thomas and F. Viviani for helpful discussions related to this work.

The research leading to these results has received funding from the European Research Council under the European Union's Seventh Framework Programme (FP7/2007-2013) / ERC Grant agreement no. 307119. GC was also funded by the grant FIRB 2012 ``Moduli Spaces and Their Applications''.

\section{Some results on filtrations in finite dimensional GIT}\label{GIT}
In this section we discuss some preliminary notions in a finite dimensional GIT context.

Let $V$ be a finite dimensional complex vector space. %In applications we will choose either $V = H^0(X,  L^{\otimes r})$ or $V=\bigwedge^d \Sym^m H^0(X, L^{\otimes r})^{\vee}$ (where $(X, L)$ is a polarised variety), but this does not matter for the time being. 
Pick an increasing filtration $F = \{ F_i V \}_{i \in \Z}$ of $V$ by complex subspaces (with index set $\Z$) and a $\C^*$-action $\lambda$ on $V$. 
\begin{definition} The specialisation $\bar{F}$ of $F$ via $\lambda$ is the filtration given by
\begin{equation*}
\bar{F}_i V = \lim_{\tau\to 0} \lambda(\tau) \cdot F_iV,
\end{equation*}
where the limit is taken in the appropriate Grassmannian. 
\end{definition}
Equivalently $\bar{F}_iV$ is the subspace spanned by the vectors $\bar{v}$ as $v$ varies in $F_iV$, where $\bar{v}$ denotes the lowest weight term with respect to the action of $\lambda$. The filtration $\bar{F}$ is $\lambda$-equivariant by construction, that is each $\bar{F}_iV$ is preserved by $\lambda$.

Let $G$ be a reductive group acting on $V$, and assume that the kernel of the action is a finite group. %In applications we will choose $G = GL(H^0(X, L^{\otimes r})^{\vee})$. 
\begin{definition} Let $\gamma$ be a one-parameter subgroup of $G$ acting on $V$ as above. The weight filtration of $\gamma$ is the increasing filtration $F = \{F_i V\}_{i \in \Z}$ given by 
\begin{equation*}
F_i V = \bigoplus_{j\geq -i}V_j
\end{equation*}
where $V_j$ is the weight $j$ eigenspace for the action of $\gamma$. %The index set is $\Z$.
\end{definition}
Let $\calP(\gamma)$ be the parabolic subgroup of $G$ associated to the one-parameter subgroup $\gamma$. By definition this is the subgroup preserving the flag $F$.

Suppose that $\lambda$ is an additional one-parameter subgroup of $G$. We wish to characterise the specialisation of the weight filtration $F$ of $\gamma$ via the action of $\lambda$. For this we recall that the intersection of parabolic subgroups $\calP(\lambda)\cap \calP(\gamma)$ contains a maximal torus $\mathcal{T}$ of $G$ (see e.g. \cite{BorelTits} Proposition 4.7). Moreover all maximal tori in a parabolic subgroup are conjugated by elements of the parabolic, hence there exists a one-parameter subgroup $\chi$ of $\mathcal{T}$ such that $\chi$ is conjugate to $\gamma$ via an element in $\calP(\gamma)$, so that the weight filtration associated to $\chi$ is still $F$. Let
\begin{equation*}
\bar{\gamma}(t) =\lim_{\tau\to 0}\lambda(\tau)\chi(t)\lambda(\tau)^{-1}.
\end{equation*}
This limit exists because $\chi$ lies in the parabolic $\calP(\lambda)$, see \cite{GIT} section 2.2.
\begin{lem}\label{spec_group}
Suppose that $F$ is the weight filtration of $\gamma$. The specialisation $\bar{F}$ of $F$ via $\lambda$ coincides with the weight filtration of $\bar{\gamma}$. It follows in particular that $\bar{F}$ is induced by a one-parameter subgroup of $G$.
\end{lem}
Note that the filtration $\bar{F}$ is uniquely defined, but $\bar{\gamma}$ is not (for example, it depends on the choice of $T$).
\begin{proof}
The key remark is that the weight $j$ eigenspace of $\lambda(\tau)\chi(t)(\lambda(\tau))^{-1}$ is $\lambda(\tau) \cdot V_j$. Now for every $v\in V$ we have
\begin{equation*}
\bar{\gamma}(t)(v)=\lim_{\tau\to 0}\lambda(\tau)\chi(t)(\lambda(\tau))^{-1}(v)
\end{equation*}
so $v$ is a weight $j$ eigenvector for $\bar{\gamma}$ if and only if $v$ belongs to
\begin{equation*}
\lim_{\tau \to 0}\lambda(\tau)\cdot V_j
\end{equation*}
where the limit is taken in the appropriate Grassmannian.
\end{proof}
\begin{definition} The Hilbert-Mumford weight of a vector $v \in V$ with respect to the one-parameter subgroup $\gamma$ is
\begin{equation*}
\HM(v,\gamma) = \min_i\{v\in F_iV\}
\end{equation*}
where $F$ is the weight filtration of $\gamma$.
\end{definition}
This depends only on the weight filtration of $\gamma$ and we will also denote it by $\HM(v,F)$ rather than $\HM(v,\gamma)$ if we wish to emphasise this fact. But notice that a general filtration of $V$ will not come from a one-parameter subgroup of the fixed reductive group $G$.\\
\begin{rmk} With our sign convention $\HM(v,\gamma)$ is the weight of the induced action of $\gamma$ on the fibre $\OO_{\PP(V)}(1)_{[v]_0}$ of the hyperplane line bundle on $\PP(V)$ over $[v]_0 = \lim_{\tau \to 0}\lambda(\tau)\cdot [v]$. Thus for example the Hilbert-Mumford Criterion says that $[v]$ is GIT semistable if and only if $\HM(v, \gamma) \geq 0$ for all one-parameter subgroups $\gamma$.
\end{rmk}

\begin{prop}\label{hm_spec}
Let $\lambda$ be a one-parameter subgroup of the stabiliser of $[v]\in \mathbb{P}(V)$. The we have
\begin{equation*}
\HM(v,\bar{F})\leq \HM(v,\gamma)
\end{equation*}
where $\bar{F}$ is the specialisation via $\lambda$ of the weight filtration $F$ of $\gamma$. 
\end{prop}
Recall that by Lemma \ref{spec_group} the filtration $\bar{F}$ is the weight filtration of a one-parameter subgroup of $G$.
\begin{proof}
We only need to show that $v\in F_iV$ implies $v\in \bar{F}_iV$. This follows from the fact that $v$ is an eigenvector of $\lambda$, so it is equal to its lowest weight term $\bar{v}$ with respect to the action of $\lambda$.
\end{proof}
It is easy to produce examples where the inequality of Proposition \ref{hm_spec} is strict.\\
\begin{exm} We choose $G = SL(2, \C)$ with its standard action on $V = \C^2$, and
\begin{equation*}
v = e_2,\, \gamma(t) = \left(\begin{matrix} t^k & 0 \\ 0 & t^{-k} \end{matrix}\right),\, \lambda(\tau) = \left(\begin{matrix} \tau^h & 0 \\ \tau^h - \tau^{-h} & \tau^{-h} \end{matrix}\right)
\end{equation*}  
for fixed $h, k > 0$. Note that $\lambda$ stabilises $[v] \in \PP(V)$. One checks that $\gamma$ is not contained in the parabolic $\calP(\lambda)$. But conjugating $\gamma$ with a suitable element in $\calP(\gamma)$ gives
\begin{equation*}
\left(\begin{matrix} 1 & 1 \\ 0 & 1\end{matrix}\right) \gamma \left(\begin{matrix} 1 & -1 \\ 0 & 1\end{matrix}\right) = \left(\begin{matrix} t^k & t^{-k} - t^k \\ 0 & t^{-k} \end{matrix}\right) = \chi \in \calP(\gamma) \cap \calP(\lambda).
\end{equation*}
A straightforward computation gives
\begin{equation*}
\lim_{\tau \to 0} \lambda(\tau) \chi (\lambda(\tau))^{-1} = \left(\begin{matrix} t^{-k} & 0 \\t^{-k} - t^k & t^k \end{matrix}\right) = \bar{\gamma},
\end{equation*}
so we have
\begin{equation*}
\HM(v, \bar{\gamma}) = -k < \HM(v, \gamma) = k. 
\end{equation*}
\end{exm}

It is important to realise that even if $\gamma$ does not stabilise $[v] \in \PP(V)$ its specialisation $\bar{\gamma}$ with respect to a $\lambda$ in the stabiliser could well lie in the stabiliser (so abusing the K-stability terminology which will be recalled in the next section, in the present finite-dimensional setup and without imposing further restrictions, we can end up with a ``product test-configuration").\\ 
\begin{exm} Let $V, \gamma, \lambda$ be as in the previous example. We choose $v = e_1 + e_2$. Then $[v] \in \PP V$ is stabilised by $\lambda$ and by $\bar{\gamma}$, but not by $\gamma$. Note that in this case we have $\HM(v, \gamma) = \HM(v, \bar{\gamma}) = k$. 
\end{exm}
%\smallskip

Let $F, F'$ be filtrations of $V$ with index set $\Z$. We say that $F$ is included in $F'$ if $F_i V \subset F'_i V$ holds for all $i$. The following observation follows immediately from the definition of the Hilbert-Mumford weight and will be useful in later applications. 
\begin{lem}\label{inclusion}
Let $F$, $F'$ be the weight filtrations of some one-parameter subgroups. If $F$ is included in $F'$ then we have
\begin{equation*}
\HM(v,F') \leq \HM(v,F)
\end{equation*}
for all $v$ in $V$.
\end{lem}

\section{Filtrations, test-configurations, approximations}\label{backSec}
Let $(X,L)$ be a polarised variety. One of the main objects of study in this paper are test-configurations of $(X,L)$. Let us briefly recall their definition.
\begin{definition}\label{testConfDef}
Let $\C^*$ act in the standard way on $\C$. A test-configuration $(\X,\LL)$ for $(X,L)$ with exponent $r$ is a $\C^*$-equivariant flat morphism $\pi\colon \X \to \mathbb{C}$, together with a $\pi$-ample line bundle $\LL$ and a linearisation of the action of $\,\C^*$ on $\LL$, such that the fibre over $1$ is isomorphic to $(\X,\LL^{\otimes r})$. We say that $(\X,\LL)$ is
\begin{itemize}
\item \emph{very ample}, if $\LL$ is $\pi$-very ample;
\item a \emph{product}, if it is isomorphic to $(X\times \C,L^{\otimes r}\boxtimes \mathcal{O}_{\C})$, where the action of $\C^*$ on $X \times \C$ is induced by a one-parameter subgroup $\lambda$ of $\Aut(X,L)$ by $\lambda(\tau)\cdot (x, z) = (\lambda(\tau)\cdot x, \tau z)$;
\item \emph{trivial}, if it is a product and, moreover, $\lambda$ is trivial;
\item \emph{normal}, if the total space $\X$ is normal;
\item \emph{equivariant with respect to a subgroup} $H \subset \Aut(X,L)$, if the action of $\C^*$ can be extended to an action of $\C^*\times H$ such that the action of $\{1\}\times H$ is the natural action of $H$ on $(X,L^{\otimes r})$;
\item in the Fano case, a test-configuration is a \emph{special degeneration} if $\X$ is normal, all the fibres are klt and a positive rational multiple of $\LL$ equals $ - K_{\X}$ (this notion is due to Tian \cite{tian}, see also \cite{chenyang} Definition 1). 
\end{itemize} 
\end{definition}
The normalisation of a test-configuration is the normalisation of $\X$ endowed with the natural induced line bundle and $\C^*$ action (or $\C^*\times H$ action). A test-configuration is a product if and only if the central fibre $\X_0$ is isomorphic to $X$: by standard theory in this case there is a trivialisation $\X \cong X \times \C$ and the $\C^*$-action on $\X$ corresponds to a $\C^*$-action on $X \times \C$ preserving $X\times\{0\}$, which must then be induced by a $\C^*$-action on $X$ as above.

The following result summarises observations of Ross-Thomas \cite{rossthomas} and Odaka \cite{odaka}.
\begin{prop}\label{filt_exponent}
For all sufficiently large $r$ there is a bijective correspondence between increasing filtrations of $H^0(X, L^{\otimes r})^{\vee}$ (with index set $\Z$)  and very ample test-configurations of exponent $r$. Such a test-configuration is a product if and only if the corresponding filtration is the weight filtration of a one-parameter subgroup of $\Aut(X,L)$, and it is equivariant with respect to a reductive subgroup $H \subset \Aut(X, L)$ if and only if the corresponding filtration is preserved by $H$.
\end{prop}
\begin{proof}
An arbitrary increasing filtration of $H^0(X,L^{\otimes r})^{\vee}$ is induced by the weight filtration of a one-parameter subgroup of $GL(H^0(X,L^{\otimes r})^{\vee})$, so we can associate to a filtration the (very ample) test-configuration induced by this one-parameter subgroup. If two one-parameter subgroups induce the same filtration then the corresponding test-configurations are isomorphic, see \cite{odaka} Theorem 2.3 and its proof. Conversely by \cite[Proposition 3.7]{rossthomas} for all sufficiently large $r$ an exponent $r$ very ample test-configuration is always induced by a one-parameter subgroup of $GL(H^0(X,L^{\otimes r})^{\vee})$, and this gives the filtration. The other claims are straightforward.
\end{proof}

%Note that the filtrations of Proposition \ref{filt_exponent} are increasing and have index set $\Z$. 
One can act on a test-configuration $(\X, \LL)$ in two basic ways (see e.g. \cite{donaldsonCalabi} section 2). Firstly we can pull-back $(\X, \LL)$ via a base-change $z\mapsto z^p$. The effect on the corresponding filtration is to multiply all the indices by $p$. Equivalently the weights of the corresponding one-parameter subgroup are multiplied by $p$. Secondly we can rescale the linearisation of the action on $\LL$ by a constant factor. The effect on the corresponding filtration is to shift all indices by some integer $k$. Equivalently we are composing the corresponding one-parameter subgroup with a one-parameter subgroup in the the center of $GL(H^0(X,L^{\otimes r})^{\vee})$, which corresponds in turn to adding $k$ to all the weights. 

Combining the two operations above we can modify the weights to get a filtration with only positive indices, or alternatively to get a filtration induced by a one-parameter subgroup of $SL(H^0(X,L^{\otimes r})^{\vee})$.

\medskip

There is a more global correspondence between filtrations and test-configu\-rations, which avoids fixing the exponent. We introduce the homogeneous coordinate ring
\begin{equation*}
R=R(X,L)=\bigoplus_{k\geq 0} R_k =\bigoplus_{k\geq 0} H^0(X,L^{\otimes k}).
\end{equation*}
We focus on filtrations of $R$ of a special type.
\begin{definition}\label{increasingfiltration} We define a \emph{filtration} $\chi$ of $R$ to be sequence of vector subspaces 
\begin{equation*}
H^0(X,\OO)=F_0R \subset F_1 R \subset \cdots
\end{equation*}
which is
\begin{itemize}
\item[(i)] \emph{exhaustive}: for every $k$ there exists a $j=j(k)$ such that $F_j R_k = H^0(X,L^{\otimes k})$,
\item[(ii)] \emph{multiplicative}: $(F_i R_l) (F_j R_m) \subset F_{i+j} R_{l+m}$, 
\item[(iii)] \emph{homogeneous}: if $f$ is in $F_iR$ then each homogeneous piece of $f$ also lies in $F_iR$.
\end{itemize} 
\end{definition}
Note that when considering filtrations of $R$ we restrict to those which only have non-negative indices. There are two basic algebraic objects attached to a filtration as above.
\begin{definition} Let $\chi$ be a filtration. The corresponding \emph{Rees algebra} is
\begin{equation*}
\rees(\chi)=\bigoplus_{i\geq 0}F_i R\,t^i
\end{equation*}
The graded modules are
\begin{equation*}
\gr_i(H^0(X,L^{\otimes k}))=F_i(H^0(X,L^{\otimes k}))/F_{i-1}(H^0(X,L^{\otimes k}))
\end{equation*}
The graded algebra is 
\begin{equation*}
\gr(\chi)=\bigoplus_{k  , i\geq 0}\gr_i(H^0(X,L^{ k}))
\end{equation*}
\end{definition}
The Rees algebra is a subalgebra of $R[t]$, and by the following elementary result it is possible to reconstruct $\chi$ from it.  

\begin{lem}\label{reconstruction}
Let $A$ be a $\C$-subalgebra of $R[t]$. We define a filtration $\chi_A$ of $R$ as follows
\begin{equation*}
F_iR =\{ s \in R \, | \, t^i s\in A\}
\end{equation*}
The filtration $\chi_A$ satisfies the conditions of Definition \ref{increasingfiltration} if and only if $A$ satisfies the conditions
\begin{itemize}
\item $A\cap R =H^0(X,\mathcal{O}_X)$;
\item for every $s\in H^0(X,L)$ there exists an $i$ such that $ t^i s \in A$;
\item if $ t^i f$ is in $A$, then, for each of the homogenous component $f_k$ of $f$, $t^i f_k$ is also in $A$.
\end{itemize}
A filtration $\chi$ equals $\chi_A$, where $A$ is the Rees algebra of $\chi$. There is an inclusion of filtrations $\chi_1\subset \chi_2$ (i.e. an inclusion of filtered pieces) if and only if there is an opposite inclusion of the corresponding Rees algebras $  \rees(\chi_2)\subset \rees(\chi_1)$.
\end{lem}

The following notion is crucial for us.
\begin{definition}\label{pol_def}
A filtration is called finitely generated if its Rees algebra is finitely generated. 
\end{definition}

Let us review the relation between finitely generated filtrations and test-configurations, as developed by Witt Nystr\"om \cite{nystrom} and Sz\'ekelyhidi \cite{gabor} (see also \cite[Proposition 2.15]{boucksom}).

Let $\chi$ be a finitely generated filtration. The Rees algebra $\rees(\chi)$ is a finitely generated flat $\C[t]$-module; this means that the associated relative Proj with its natural $\OO(1)$ is a test-configuration $(\X,\LL)$. The central fibre is the Proj of the graded algebra $\gr(\chi)$; the $\C^*$-action on the central fibre is given by \emph{minus} the $i$-grading of $\gr(\chi)$.% \todo{credo che l'ultima frase sia sbagliata}.

On the other hand let $(\X,\LL)$ be an exponent $r$ test-configuration. Consider the filtration $F$ of $H^0(X, L^{\otimes r})$ associated to it by Proposition \ref{filt_exponent}. Up to base-change and scaling of the linearisation we can assume that all the weights are positive. Denote by $N$ the length of this filtration. Let $A$ be the $\C$-subalgebra of $R[t]$ generated by
\begin{equation*}
H^0(X,L)t^N \oplus \bigoplus_{i=1}^N F_iH^0(X, L^{\otimes r})t^i
\end{equation*}
Then the filtration associated to $A$ via Lemma \ref{reconstruction} is the filtration of $R$ induced by $(\X,\LL)$ (the second assumption in Lemma \ref{reconstruction} holds because $L$ is projectively normal, i.e. $R$ is generated in degree $1$).

Suppose that $\chi$ is a not necessarily finitely generated filtration. Following \cite{gabor} we can define finitely generated approximations $\chi^{(r)}$ as follows. Let $F$ be the filtration induced by $\chi$ on $H^0(X, L^{\otimes r})$, this corresponds to an exponent $r$ test-configuration $(\X,\LL)$, then $\chi^{(r)}$ is the finitely generated filtration corresponding to $(\X,\LL)$. Note that this construction also makes sense when $\chi$ is finitely generated and corresponds to $(\X,\LL)$, in which case $\chi^{(r)}$ corresponds to $(\X,\LL^{\otimes r})$.\\

%Given a filtration one can define the corresponding weight functions.
\begin{definition}\label{weight_def}
We introduce two ``weight functions" attached to $\chi$, given by
\begin{equation*}
w_{\chi}(k)=w(k)=\sum_i (-i) \dim \gr_i(H^0(X,L^{\otimes k})),
\end{equation*}
respectively
\begin{equation*}
d_{\chi}(k)=d(k)=\sum_i i^2 \dim \gr_i(H^0(X, L^{\otimes k})).
\end{equation*}
\end{definition}
If $\chi$ is a finitely generated filtration (corresponding to a test-configuration $(\X, \LL)$) then by equivariant Riemann-Roch we have, for all sufficiently large $k$,
\begin{align*}
h(k) &= h^0(X, L^{\otimes k}) = a_0 k^n + a_1k^{n-1}+ \cdots \\
w(k) &=b_0 k^{n+1} + b_1k^n+ \cdots \\
d(k) &=c_0 k^{n+2} + c_1k^{n+1}+ \cdots 
\end{align*}
\begin{definition}\label{df_def}
Let $\chi$ be a finitely generated filtration (which thus corresponds to a test-configuration). One defines the $r$-th Chow weight, Donaldson-Futaki weight (or invariant) and the $L^2$ norm as
\begin{align*}
\Chow_r(\chi)&=\Chow_r(\X,\LL)=r\frac{b_0}{a_0}-\frac{w(r)}{d(r)},\\
\df(\chi) &= \df(\X, \LL) = \frac{a_1b_0 - a_0b_1}{a_0^2},\\
||\chi ||_{L^2}^2 &= ||(\X, \LL)||_{L^2} = c_0-\frac{b_0^2}{a_0}.
\end{align*}
\end{definition}
Note that a straightforward computation shows that we have
\begin{equation*}
\lim_{r\to \infty} \Chow_r(\X, \LL^{\otimes r}) = \df(\X, \LL).
\end{equation*}
%Remark that the Donaldson-Futaki weight is the limit for $r$ that goes to infinity of the Chow weights.
\begin{definition} A polarised variety $(X,L)$ is K-semistable if $\df(\X,\LL)\geq 0$ for every test-configuration $(\X,\LL)$.

Given a subgroup $H$ of $\Aut(X,L)$, we say that $(X,L)$ is $H$-equivariantly $K$-semistable if $\df(\X,\LL)\geq 0$ for every $H$-equivariant test-configuration $(\X,\LL)$.
\end{definition}

\begin{definition}\label{KstabDef} A normal polarised variety $(X,L)$ is K-polystable if for every test-configuration $(\X,\LL)$ with normal total space we have $\df(\X,\LL)\geq 0$, with equality if and only if $(\X,\LL)$ is a product. 

Given a subgroup $H$ of $\Aut(X,L)$, $(X, L)$ is $H$-equivariantly K-polystable if for every $H$-equivariant test-configuration $(\X,\LL)$ with normal total space we have $\df(\X,\LL)\geq 0$, with equality if and only if $(\X,\LL)$ is a product.

\end{definition}

Following \cite{gabor} (Definition 3 and Equation (33)) we also define the following two invariants of a non-finitely generated filtration.
\begin{definition}\label{ChowInftyDef}
The Donaldson-Futaki and asymptotic Chow weights of a filtration $\chi$ are given by
\begin{equation*}
\df(\chi) = \liminf_{r \to \infty} \df(\chi^{(r)})\, ,
\end{equation*}
respectively
\begin{equation*}
\Chow_{\infty}(\chi) = \liminf_{r \to \infty} \Chow_r(\chi^{(r)})\,.
\end{equation*}

\end{definition}
Note that $\chi^{(r)}$ is an exponent $r$ test configuration, so it is natural to consider its $r$-th Chow weight. Let us also emphasise that, when $\chi$ is finitely generated, both these invariants coincide with the classical Donaldson-Futaki weight, see \cite[Section 3.2]{gabor}. In general these two invariants differ, see \cite[Example 4]{gabor}; we do not know if there is an inequality relating them.

\begin{definition}\label{def:norm_inf}
The $L^2$ norm of a filtration $\chi$ is given by
\begin{equation*}
||\chi||_{2} = \liminf_{r \to \infty}||\chi^{(r)}||\,.
\end{equation*}
\end{definition}
In \cite[Lemma 8]{gabor} it is shown that the above liminf is actually a limit.

\begin{definition}\label{def:cappuccio}
A polarised variety is $\hat{K}$-semistable if for any filtration $\chi$ of $R(X,L)$ we have
$$\df(\chi) \geq 0.$$
It is  $\hat{K}$-stable if the equality holds if and only if $||\chi||_2=0$. One can make parallel definitions replacing $\df(\chi)$ with the asymptotic Chow weight $\Chow_{\infty}(\chi)$.
%We say that it is $\hat{\Chow}$-semistable if
%$$ \Chow_{\infty}(\chi)\geq 0.$$
\end{definition}
One easily checks that $\hat{K}$-semistability is equivalent to K-semistability. On the other hand $\hat{K}$-stability is (at least a priori) stronger than $K$-stability, and just as $K$-stability it implies that the automorphism group of $(X,L)$ has no nontrivial one-parameter subgroups. 

Sz\'ekelyhidi \cite{gabor} (Theorem 10 and Proposition 11) proves that if $(X, L)$ is cscK with trivial automorphisms then it is $\hat{K}$-stable, including the variant notion using the $\Chow_{\infty}$ weight.

At present we do not  know a good candidate for the notion of $\hat{K}$-poly\-stability (i.e. allowing continuous automorphisms of $X$).

%\begin{rmk}
%If we take positive tensor powers of the polarisation $L$, the stability of the polarised variety $(X,L)$. Because of this, there is no difference in assuming that $L$ is just ample or rather very ample and projectively normal.
%\end{rmk}

\section{Specialisation of a test-configuration}\label{SPEC}
In the classical situation of a torus $T$ acting on a projective variety one can specialise a point $p$ to a fixed point $\bar{p}$ for the action of $T$: one picks a generic one-parameter subgroup $\lambda$ of $T$ and the specialisation is $\bar{p} =\lim_{\tau \to 0}\lambda(t)\cdot p$. This specialisation does depend on $\lambda$ and when we need to emphasise this dependence we will denote it by $\bar{p}_{\lambda}$. In this section we first generalise this construction to test-configurations, and then prove some basic facts which imply our main result Theorem \ref{mainThm}.

\begin{definition}\label{specTCDef} Let $(\X,\LL)$ be an exponent $r$ test-configuration and $F$ be the corresponding filtration of $H^0(X, L^{\otimes r})^{\vee}$ given by Proposition \ref{filt_exponent}. Let $T$ be a torus in $\Aut(X,L)$, and $\bar{F}$ the specialisation of $F$ via a generic one-parameter subgroup $\lambda$ of $T$. Then  the specialisation $(\bar{\X},\bar{\LL})$ of $(\X,\LL)$ is the $T$-equivariant exponent $r$ test-configuration corresponding to $\bar{F}$.
\end{definition}
The specialisation depends on the choice of $r$ and $\lambda$, but we will mostly suppress this in the notation. 

We make a brief digression in order to discuss Definition \ref{specTCDef}. Recall that by Proposition \ref{filt_exponent} an exponent $r$ test-configuration for $(X, L)$ is obtained by embedding $\iota\!: X \hookrightarrow \PP H^0(X, L^{\otimes r})^{\vee}$ with the complete linear system $| r L |$ and by taking the flat closure of $\iota(X)$ under the action of a one-parameter subgroup $\gamma$ of $GL(H^0(X,  L^{\otimes r})^{\vee})$. The corresponding test-configuration $(\X, \LL)$ is a closed subscheme of $\PP H^0(X,  L^{\otimes r})^{\vee} \times \C$ (in fact it can be canonically completed to a closed subscheme of $\PP H^0(X, L^{\otimes r})^{\vee} \times \PP^1$ by gluing with the trivial family at infinity). If $\lambda$ is a one-parameter subgroup of $\Aut(X, L)$ one could attempt to define the $\lambda$-specialisation of $(\X, \LL)$ by taking its flat closure as a closed subscheme of $\PP H^0(X, L^{\otimes r})^{\vee} \times \C$ under the action of $\lambda$. We give a simple example showing that such a flat closure is not preserved by $\gamma$ in general, so it is not a $\lambda$-equivariant test-configuration in a natural way. In fact we also show that in general the total space of the flat closure cannot support a test-configuration, and compute the corresponding specialisation $(\bar{\X}, \bar{\LL})$ in the sense of Definition \ref{specTCDef} in the example.\\ %and moreover that the specialisation $(\bar{\X}, \bar{\LL})$ of $(\X, \LL)$ in the sense of Definition \ref{specTCDef} is not in general the central fibre of a $\C^*$-equivariant flat degeneration of $(\X, \LL)$.\\
\begin{exm} Embed $\iota\!: \PP^1 \hookrightarrow \PP^2$ via Veronese $[s_0 : s_1] \mapsto [s^2_0 :  s_0 s_1 : s^2_1]$ and act with the one-parameter subgroup $\gamma$ of $SL(3, \C)$ given by $\operatorname{diag}(t^{-1}, t^2, t^{-1})$. This gives a test-configuration $(\X, \LL)$ of exponent $2$ for $(\PP^1, \OO_{\PP^1}(1))$ with total space $\X \subset \PP^2 \times \C$ which is the variety $V( x z - t^6 y^2 )$. Now choose 
\begin{equation*}
\lambda = \left(\begin{matrix} 1 & 1 \\ 0 & 1 \end{matrix}\right) \left(\begin{matrix} \tau^h & 0 \\ 0 & \tau^{-h}\end{matrix}\right)\left( \begin{matrix} 1 & - 1 \\ 0 & 1 \end{matrix} \right) \in SL(2, \C) = \Aut(\PP^1, \OO_{\PP^1}(1)).
\end{equation*}
The induced one-parameter subgroup in $SL(3, \C)$, which we still denote by $\lambda$, is given by
\begin{equation*}
\lambda = \left(\begin{matrix} \tau^{2h} & 1 - \tau^{2h} & (\tau^{-h} - \tau^h)^2 \\ 0 & 1 & -2(1 - \tau^{-2h}) \\ 0 & 0 & \tau^{-2h} \end{matrix}\right). 
\end{equation*} 
One computes
\begin{equation*}
\lambda(\tau) \cdot \mathcal{X} = V(\tau^{2h} x ((\tau^{-h} - \tau^h)^2 x - 2 (1 - \tau^{-2h})y + \tau^{-2h} z) - t^6 ((1 - \tau^{2h}) x + y)^2). 
\end{equation*}
Since $\lambda(\tau) \cdot \mathcal{X} \subset \PP^2 \times \C$ is a familiy of divisors it is straightforward to take the flat limit at $\tau \to 0$. For $h > 0$ one finds 
\begin{equation}\label{exampleTC}
\lim_{\tau \to 0} \lambda(\tau) \cdot \mathcal{X} = V(x (x + 2y + z) - t^6 (x + y)^2).
\end{equation}
The central fibre $V(x (x + 2y + z))$ is not preserved by $\gamma$ so the flat limit is not a $\lambda$-equivariant test-configuration in a natural way, although in this case it does support some other non-canonical $\lambda$-equivariant test-configuration. On the other hand for $h < 0$ we find that the flat limit is given by the divisor
\begin{equation*}
\lim_{\tau \to 0} \lambda(\tau) \cdot \mathcal{X} = V( x^2 (t^6 - 1)).
\end{equation*}
This may be thought of as the product, thickened test-configuration $V( x^2 )$ glued to six copies of $\PP^2$, and clearly it cannot be the total space of a test-configuration for $\PP^1$.  

We can also consider the specialisation $(\bar{\X}, \bar{\LL})$ of $(\X, \LL)$ in the sense of Definition \ref{specTCDef}. The conjugate one-parameter subgroup $\lambda(\tau) \gamma(t) (\lambda(\tau))^{-1}$ is given by 
\begin{equation*}
\left(\begin{matrix} t^{-1} &  -t^{-1}(-1 + \tau^{2h})(-1 + t^3) & -2 t^{-1}(-1 + \tau^{2h})^2 (-1 + t^3)\\
0 & t^2 & 2 t^{-1} (-1 + \tau^{2h})(-1 + t^3)\\
0 & 0 & t^{-1}
\end{matrix}\right),  
\end{equation*}
so $\gamma$ lies in the parabolic $\calP(\lambda)$ if and only if $h > 0$. In this case $(\bar{\X}, \bar{\LL})$ is obtained by acting on $V(x z - y^2)$ with $\bar{\gamma} = \lim_{\tau \to 0} \lambda(\tau) \gamma(t) (\lambda(\tau))^{-1}$. The resulting test-configuration is precisely \eqref{exampleTC}. The central fibre $\bar{\X}_0 = V(x(2(x+y) + z))$ is preserved by $\bar{\gamma}$ and $\lambda$ and we obtain a $\lambda$-equivariant test-configuration in a canonical way. 

For $h < 0$ we have $\gamma \notin \calP(\lambda)$ and we must first conjugate $\gamma$ by some element $g \in \calP(\gamma)$ to obtain $\chi \in \calP(\lambda)$. A direct computation shows that one can choose 
\begin{equation*}
g = \left(\begin{matrix} 1 & 0 & -1 \\ 1 & 1 & 0\\ 0 & 0 & 1\end{matrix}\right),\quad \chi = \left(\begin{matrix} t^{-1} & 0 & 0\\ t^{-1} - t^2 & t^2 & t^{-1} - t^2 \\ 0 & 0 & t^{-1}\\  \end{matrix}\right)
\end{equation*}   
yielding
\begin{equation*}
\bar{\gamma} = \lim_{\tau \to 0} \lambda(\tau) \chi(t) (\lambda(\tau))^{-1} = \left(\begin{matrix} t^2 & t^{-1} - t^2 & -t^{-1} + t^2 \\ 0 & t^{-1} & 0 \\ 0 & 0 & t^{-1}\end{matrix}\right).
\end{equation*}
The corresponding test-configuration $(\bar{\X}, \bar{\LL})$ is given by
\begin{equation*}
V( t^3 x (x + 2y +z) - (x+y)^2)
\end{equation*}
endowed with the action of $\bar{\gamma}$, which commutes with $\lambda$. Diagonalising $\bar{\gamma}$ (which is of course compatible with diagonalising $\lambda$) we see that $(\bar{\X}, \bar{\LL})$ is isomorphic to the test-configuration induced by $\operatorname{diag}(t^{-1}, t^{-1}, t^2)$ given by $V(t^3 x z - y^2)$. 

Finally note that the test-configuration $(\X', \LL')$ (isomorphic to $(\X, \LL)$) defined by $\chi$ is
\begin{equation*}
V((x+y)(y+z) - t^3 y (x +2y + z)).
\end{equation*}
Taking the flat closure of $(\X', \LL')$ under the action of $\lambda$ gives the one-parameter family of divisors of $\PP^1 \times \C$ parametrised by $\tau$
\begin{equation*}
(x + y)^2 - t^3 x (x + 2 y + z) + \tau^{-2h} (1 - t^3)(x+y)(x + 2y + z).
\end{equation*} 
This is a flat one-parameter family taking $(\X', \LL')$ to $(\bar{\X}, \bar{\LL})$.
\end{exm}

We explain next an alternative approach to specialising test-configurations which is more global, i.e. independent of the exponent, and is based on filtrations of the homogeneous coordinate ring. Let $\chi$ be the filtration of $R = R(X, L)$ corresponding to $(\X,\LL)$. 
\begin{definition}\label{specFILTDef} Let $\lambda\!: \C^* \to T$ be a generic one-parameter subgroup. The specialisation $\bar{\chi}$ of $\chi$ with respect to $T$ is given by $\bar{\chi}_i = \lim_{\tau \to 0} \lambda(\tau)\cdot \chi_i$, where the limit is taken in the appropriate Grassmannian.
\end{definition}
It is straightforward to check that $\bar{\chi}$ is still a filtration of $R$ in the sense of Definition \ref{increasingfiltration}. The limit filtration $\bar{\chi}$ can also be described as follows. Let $\rees(\chi) \subset R$ be the Rees algebra of the finitely generated filtration $\chi$. A one-parameter subgroup $\lambda\!: \C^* \to \Aut(X, L)$ acts on $R$ and on $R[t]$ (trivially on $t$) and we may define a $\C[t]$-subalgebra $\rees^{\lambda}(\chi) \subset R$ by
\begin{equation*}
\rees^{\lambda}(\chi)  = \{ \lim_{\tau \to 0} \lambda(\tau)(s) \,:\, s\in \rees(\chi)\}.
\end{equation*}
Then $\bar{\chi}$ is precisely the filtration of $R$ whose Rees algebra is $\rees^{\lambda}(\chi)$, i.e.
\begin{equation*}
\bar{F}_i R_k = \{ s \in R_k \,:\, t^i s \in \rees^{\lambda}(\chi)\}.
\end{equation*}

The crucial difficulty with this more global approach lies in the fact that the Rees algebra of $\bar{\chi}$ is not finitely generated in general. This is a well-known phenomenon in commutative algebra and an explicit example is given in the Appendix. 

Let $(\X, \LL)$ be a very ample test-configuration of exponent $r$. Given a generic one-parameter subgroup of $T \subset \Aut(X, L)$ we can perform two basic constructions. On the one hand we can specialise $(\X, \LL)$ to $(\bar{\X}, \bar{\LL})$ in the sense of Definition \ref{specTCDef}. This specialisation corresponds to a finitely generated filtration $\eta$. The Veronese filtration $\eta^{(j)}$ corresponds to the Veronese test-configuration $(\bar{\X}, \bar{\LL}^{\otimes j})$ with exponent $j r$. On the other hand $(\X, \LL)$ corresponds to a finitely generated filtration $\chi$ of $R$ via the construction described at the end of the previous section. We may specialise $\chi$ to $\bar{\chi}$ and consider a finitely generated approximation $\bar{\chi}^{(j)}$, corresponding to a test-configuration of exponent $j r$: by definition this is in fact $(\bar{\X}, \overline{\LL^{\otimes j}})$. Since $\bar{\chi}$ is not finitely generated (in general), the filtrations $\eta^{(j)}$, $\bar{\chi}^{(j)}$ will differ for infinitely many $j$, that is the test-configurations $(\bar{\X}, \bar{\LL}^{\otimes j})$ and $(\bar{\X}, \overline{\LL^{\otimes j}})$ differ for infinitely many $j$. However we can establish a simple comparison.

\begin{prop}\label{inclusion2}
The filtration of $H^0(X, L^{\otimes jr})$ induced by $\bar{\chi}$ (or equivalently by $\bar{\chi}^{(j)}$ or $(\bar{\X}, \overline{\LL^{\otimes j}})$) is included in the filtration of the same vector space induced by $\eta^{(j)}$, i.e. by the filtration corresponding to $(\bar{\X}, \bar{\LL}^{\otimes j })$.  
\end{prop}
\begin{proof}
The result follows at once from the fact that the Rees algebra of $\bar{\chi}$ contains all the generators of the Rees algebra of $\eta$, by construction.
\end{proof}

Let us show that when $\bar{\chi}$ is finitely generated then $(\bar{\X}, \bar{\LL})$ is in fact a flat limit of $(\X, \LL)$ under a $\C^*$-action, and in particular the filtrations $\bar{\chi}^{(j)}$, $\eta^{(j)}$ coincide for all $j$, that is $(\bar{\X}, \bar{\LL}^{\otimes j})$ and $(\bar{\X}, \overline{\LL^{\otimes j}})$ coincide. In order to simplify the notation (without loss of generality) we assume in the following result that $(\bar{\X}, \bar{\LL})$ has exponent $1$ and $\chi$ is the corresponding finitely generated filtration. 

\begin{lem}\label{familyLem} Suppose that $\rees(\bar{\chi}) = \rees^{\lambda}(\chi)$ is a finitely generated $\C[t]$-subalgebra of $R[t]$. Then there exist an embedding $\iota\!: \X \to \PP^N \times \C$ and a 1-parameter subgroup $\widehat{\lambda}\!:\C^* \to GL(N+1, \C)$ such that 
\begin{enumerate}
\item[$\bullet$] $\iota^* \OO_{\PP^N}(1) = \LL^{\otimes r}$ for some $r \geq 1$, 
\item[$\bullet$] $\widehat{\lambda}$ acting on $\PP^N$ preserves $\iota(\X_1) \cong X$ and restricts to the induced action of $\lambda$ on it,
\item[$\bullet$] the 1-parameter flat family of subschemes of $\PP^N \times \C$ induced by $\widehat{\lambda}$ (acting trivially on the second factor) has central fibre isomorphic to $\proj(\rees(\bar{\chi}))$ endowed with its natural Serre line bundle $\OO(r)$. 
\end{enumerate}
In particular it follows that the central fibre $(\X'_0, \LL'^{\otimes r}_0)$ is a flat 1-parameter degeneration of the central fibre $(\X_0, \LL^{\otimes r}_0)$ (as closed subschemes of $\PP^N$).
\end{lem} 
\begin{proof} If $\rees(\bar{\chi}) = \rees^{\lambda}(\chi) \subset R[t]$ is a finitely generated $\C[t]$-subalgebra there exists a finite set of elements $\sigma_i$ of $\rees(\chi)$ such that the limits $\lim_{\tau \to 0} \lambda(\tau) \cdot \sigma_i$ generate $\rees(\bar{\chi})$. Since $\lambda(\tau)$ is $\C[t]$-linear and we have $\lambda(\tau)\cdot(s_1 + s_2) = \lambda(\tau) \cdot s_1 + \lambda(\tau) \cdot s_2$ and $\lambda(\tau)\cdot (s_1 s_2) = (\lambda(\tau)\cdot s_1)(\lambda(\tau)\cdot s_2)$ for all $s_1, s_2 \in R$ we can choose our $\sigma_i$ of the special form $\sigma_i = t^{p(i)} s_i$ where the $s_i$ are homogeneous elements of $R$. Moreover we can assume that the elements $t^{p(i)} s_i$, $i = 0, \ldots, N$ generate $\rees(\chi)$. For a suitable $r \geq 1$ the monomials $\tilde{s}_j$ in our elements $s_i$ of homogenous degree $r$ generate the Veronese algebra $\tilde{R} = \bigoplus_{k \geqslant 0} R_{k r}$ (which is thus generated in degree $1$) and so the corresponding elements $t^{p(j)} \tilde{s}_j$ generate the Veronese algebra $\bigoplus_{k \geqslant 0} (F_{k r} \tilde{R}) t^{k r}$ and their limits $t^{p(j)} \lim_{\tau \to 0} \lambda(\tau) \cdot \tilde{s}_j$ generate the Veronese algebra $\bigoplus_{k \geqslant 0} (\bar{F}_{k r} \tilde{R}) t^{k r}$. 

With these assumptions we define a surjective morphism of $\C[t]$-algebras 
\begin{equation*}
\phi\!: \C[\xi_0, \ldots, \xi_N][t] \to \bigoplus_{k \geqslant 0} (F_{k r} \tilde{R}) t^{k r}
\end{equation*}
by $\phi(t) = t$, $\phi(\xi_i) = t^{p(i)} \tilde{s}_i$. Suppose that the action of $\lambda$ is given by $\lambda(\tau) \cdot \tilde{s}_i = \sum_j  a_{ij}(\tau) \tilde{s}_j$. We define a one-parameter subgroup $\widehat{\lambda}\!: \C^* \to GL(\C_1[\xi_0, \ldots, \xi_N])$, acting on degree $1$ elements by $\widehat{\lambda}(\tau)\cdot\xi_i = \sum_j a_{ij}(\tau)\xi_j$, and extend its action trivially on $t$. The morphism $\phi$ induces the required embedding 
\begin{equation*}
\iota\!: \X = \proj_{\C[t]} \bigoplus_{k \geqslant 0} (F_{k r} \tilde{R}) t^{k r} \to \proj_{\C[t]} \C[\xi_0, \ldots, \xi_N][t],  
\end{equation*} 
which intertwines the actions of $\lambda$ and $\widehat{\lambda}$. By construction the limit as $\tau \to 0$ of the flat family of closed subschemes of $\PP^N \times \C$ given by \begin{equation*}
\widehat{\lambda}(\tau) \cdot \iota\big(\proj_{\C[t]} \bigoplus_{k \geqslant 0} (F_{k r} \tilde{R}) t^{k r}\big)
\end{equation*} is isomorphic to $\proj_{\C[t]} \bigoplus_{k \geqslant 0} (\bar{F}_{k r} \tilde{R}) t^{k r}$ and so it gives a copy of $\X'$ embedded in $\PP^N \times \C$ as a flat $1$-parameter degeneration of $\X$. 

To prove the statement on central fibres we look at the family of closed subschemes of $\PP^N$ given by 
\begin{equation*}
\widehat{\lambda}(\tau) \cdot \iota(\X_0) = \widehat{\lambda}(\tau) \cdot \iota\big(\proj_{\C[t]} \gr \bigoplus_{k \geqslant 0} (F_{k r} \tilde{R}) t^{k r}\big).
\end{equation*}
Taking the flat closure of this 1-parameter family we obtain a closed subscheme $\Y_0 \subset \PP^N$ whose underlying reduced subscheme $\Y^{\red}_0$ is contained in $\X'_0 \subset \PP^N$. By flatness the Hilbert function of $\Y_0$ is the same as that of the central fibre $(\X_0, \LL^{\otimes r}_0)$ and so the same as that of the general fibre $(X, L^{\otimes r})$. Similarly the Hilbert function of $\X'_0 \subset \PP^N$ is the same as that of $(\X'_0, \LL'^{\otimes r}_0)$ and so the same as that of the general fibre $(X, L^r)$. As we have $\Y^{\red}_0 \subset \X'_0 \subset \PP^N$ and $\X'_0, \Y_0 \subset \PP^N$ have the same Hilbert functions we must actually have $\Y'_0 = \X'_0$ as required.
\end{proof}
The following observation follows immediately from the definitions of the weight functions (Definitions \ref{weight_def}, \ref{df_def}) and of the specialisation $\bar{\chi}$ (Definition \ref{specFILTDef}).
\begin{lem}\label{familyLem2} In the situation of Lemma \ref{familyLem} we have 
\begin{equation*}
w_{(\bar{\X}, \bar{\LL})}(k) = w_{(\X, \LL)}(k),\quad d_{(\bar{\X}, \bar{\LL})}(k) = d_{(\X, \LL)}(k).
\end{equation*}
for all $k$. In particular we have
\begin{equation*}
\df(\bar{\X}, \bar{\LL}) = \df(\X, \LL), \quad || (\bar{\X}, \bar{\LL}) ||_{L^2} = ||(\X, \LL)||_{L^2}.
\end{equation*}
\end{lem}
Let us now consider the general case.
\begin{thm}\label{main}
Let $\chi$ be a possibly non-finitely generated filtration, and let $\bar{\chi}$ be its specialisation with respect to a torus $T\subset \Aut(X,L)$ in the sense of Definition \ref{specFILTDef}. Then we have
\begin{equation*}
\Chow_{\infty}(\bar{\chi})\leq \Chow_{\infty}(\chi).
\end{equation*}
\end{thm}
\begin{proof}
We claim that the inequality $\Chow_r(\bar{\chi}^{(r)})\leq \Chow_r(\chi^{(r)})$ holds for every $r$. By Definition \ref{ChowInftyDef} this would imply the Theorem. 

Before proving the claim, let us recall the relation between the Chow weight and classical GIT, following \cite[Section 3]{rossthomas}, \cite[Section 7]{Futaki} and \cite{gabor}. Let $V_r=H^0(X,L^{\otimes r})^{\vee}$, and denote by $\gamma$ a 1PS of $GL(V_r)$ which induces the test configuration associated to $\chi^{(r)}$. The group $GL(V_r)$ acts on the appropriate Chow variety $Z_r$, and $X \subset \PP(H^0(X,L^{\otimes r})^{\vee})$ gives a point $[X] \in Z_r$. On $Z_r$ we have the classical, ample Chow line bundle, giving a linearisation for the action of $GL(V_r)$. The $r$-th Chow weight of $\chi^{(r)}$ introduced in Definition \ref{df_def} is the Hilbert-Mumford weight of the point $[X]\in Z_r$ under $\gamma$, computed with respect to a convenient rational rescaling of the ample Chow line bundle (with this normalisation the Chow line bundle becomes an ample $\mathbb{Q}$-line bundle, but this causes no difficulties).

The claim now follows from Proposition \ref{hm_spec}, i.e. the fact that Hilbert-Mumford weights decrease under specialisation.
\end{proof}

\begin{section}{Application to cscK polarised manifolds}\label{CSCK}
In this Section we show that Conjecture \ref{conj} combined with ideas from \cite{jacopo,gaborJ} implies a new proof that cscK manifolds are K-polystable.
\begin{thm}\label{thm:cscK}
Let $(X,L)$ be a cscK manifold and let $T$ be a maximal torus in $\Aut(X,L)$. Then $(X,L)$ is $T$-equivariantly K-polystable. More explicitly, given a normal $T$-equivariant test configuration $(\X,\LL)$, we have
\begin{equation*}
\df(\X,\LL)\geq 0
\end{equation*}
with equality if and only if $(\X,\LL)$ is a product.
\end{thm}
\begin{proof}

By a result of Donaldson  \cite{donaldsonCalabi} $(X, L)$ is K-semistable, so it is enough to assume that $(\X,\LL)$ is not a product and to show that we cannot have $\df(\X, \LL) = 0$. We argue by contradiction assuming $\df(\X, \LL) = 0$. 

Denote by $\alpha$ the $\C^*$ action on $(\X,\LL)$.  As $(\X, \LL)$ is $T$-equivariant, there are $\C^*$-actions $\tilde{\beta}_i$ on $(\X, \LL)$, preserving the fibres, commuting with each other and with $\alpha$, and extending the action of an orthogonal basis of 1-parameter subgroups $\beta_i$ of $\Aut(X, L)$ (see \cite{gaborExtr} for a discussion of the formal inner product on $\C^*$-actions). Fixing $i$, the total space $(\X, \LL)$ endowed with the $\C^*$-action $\alpha \pm \tilde{\beta}_i$ is a test-configuration for $(X, L)$, with Donaldson-Futaki invariant 
\begin{align*}
\df(\alpha \pm \tilde{\beta}_i) &= \df(\alpha) \pm \df(\tilde{\beta}_i)\\ 
&= \pm \df(\tilde{\beta}_i)
\end{align*} 
(the first equality follows since $\alpha$, $\tilde{\beta}_i$ are commuting $\C^*$-actions on the same polarised scheme). Since we are assuming that $(X, L)$ is cscK we know it is K-semistable and so we must have $\df(\tilde{\beta}_i) = 0$ for all $i$. Let $(\X, \LL)^{\perp}_T$ denote the $L^2$-orthogonal in the sense of \cite{gaborExtr}, i.e. the test-configuration with total space $(\X, \LL)$ endowed with $\C^*$-action 
\begin{equation*}
\alpha - \sum_i \frac{\bra \alpha, \tilde{\beta}_i \ket}{|| \tilde{\beta}_i||^2} \tilde{\beta}_i. 
\end{equation*} 
Then we see that $\df(\X, \LL)^{\perp}_T  = 0$.

Since $\X$ is normal and not isomorphic to $X\times \C$, by \cite{gaborJ} section 3 there exists a point $p \in (\X_1, \LL_1)$ which is fixed by the maximal torus $T$, and such that denoting by $\overline{\alpha \cdot p}$ the closure of the orbit of $p$ in $(\X, \LL)$ we have
\begin{align}\label{blpFormula}
\nonumber \df(\operatorname{Bl}_{\overline{\alpha \cdot p}} \X, \LL - \epsilon \mathcal{E})^\perp_T &= \df(\X, \LL)^{\perp}_T - C \epsilon^{n - 1} + O(\epsilon^n)\\
&= - C \epsilon^{n - 1} + O(\epsilon^n)
\end{align}
for some constant $C > 0$. Here $(\operatorname{Bl}_{\overline{\alpha \cdot p}} \X, \LL - \epsilon \mathcal{E})$ is the test-configuration for $(\operatorname{Bl}_p X, L - \epsilon E)$ ($E, \mathcal{E}$ denoting the exceptional divisors) induced by blowing up the orbit $\overline{\alpha \cdot p}$ in $\X$ with sufficiently small rational parameter $\epsilon > 0$. Since $p$ is fixed by $T$ there is a natural inclusion $T \subset \Aut(\operatorname{Bl}_p X, L - \epsilon E)$ and then $(\operatorname{Bl}_{\overline{\alpha \cdot p}} \X, \LL - \epsilon \mathcal{E})^\perp_T$ denotes the $L^2$ orthogonal to $T$ in the sense of \cite{gaborExtr}.

As explained in \cite{gaborJ} Theorem 2.4  a well-known result of Arezzo, Pacard and Singer \cite{arezzo} implies that the polarised manifold $(\operatorname{Bl}_p X, L - \epsilon E)$ admits an extremal metric in the sense of Calabi. The semistability result of \cite{gaborExtr} shows that we must have $\df(\operatorname{Bl}_{\overline{\alpha \cdot p}} \X, \LL - \epsilon \mathcal{E})^\perp_T \geq 0$. But this contradicts \eqref{blpFormula}, so we must have in fact $\df(\X, \LL) > 0$ as claimed.
\end{proof}

\begin{cor}
If Conjecture \ref{conj} holds, then cscK manifolds are K-polystable.
\end{cor}
\begin{proof}
Let $(X,L)$ be a cscK manifold, and $T$ a maximal torus in $\Aut(X,L)$. Theorem \ref{thm:cscK} implies that $(X,L)$ is $T$-equivariantly K-polystable. Conjecture \ref{conj} then implies that $(X,L)$ is K-polystable.
\end{proof}

\begin{rmk} The proof of the main result of \cite{gaborJ} (Theorem 1.4) shows that if $(X, L)$ is extremal and $T \subset \Aut(X, L)$ is a maximal torus then we have $\df(\X, \LL)^{\perp}_T > 0$ for all $T$-equivariant test-configurations whose normalisation is not induced by a holomorphic vector field in $T$ (or equivalently, which are not isomorphic to such a product outside a closed subscheme of codimension at least $2$). If the assumption is dropped there are counterexamples. Note that Theorem 1.4 in \cite{gaborJ} is mistakenly stated without this assumption. See \cite{chenyang} Remark 4  and the note \cite{jacopoErratum} for further discussion.
\end{rmk}
\end{section}

\appendix
\setcounter{secnumdepth}{0}
\section{Appendix}
  
In this appendix we present an example of a test-configuration $(\X, \LL)$ with a 1-parameter subgroup $\lambda\!: \C^* \to \Aut(X, L)$ such that the $\lambda$-equivariant filtration $\bar{\chi}$ of Definition \ref{specTCDef} is not finitely generated. This is done by adapting a well-known example in the literature on canonical bases of subalgebras, due to Robbiano and Sweedler (\cite{robbiano} Example 1.20). 

Consider the polynomial algebra $\C[t][x, y]$ over the ring $\C[t]$ and let $A$ denote the $\C[t]$-subalgebra generated by 
\begin{equation*}
t(x + y), t xy, t xy^2, t^2 y.
\end{equation*}
Then $A \subset R[t]$ is the Rees algebra of a homogeneous, multiplicative, pointwise left bounded finitely generated filtration $\chi$ of the homogeneous coordinate ring $R = \C[x, y]$ of the projective line $(\PP^1, \OO_{\PP^1}(1))$. So $\proj_{\C[t]} A$ endowed with its natural Serre bundle $\OO(1)$ is a test-configuration for $\PP^1$. Consider the 1-parameter subgroup $\lambda\!: \C^* \to SL(H^0(\PP^1, \OO_{\PP^1}(1)))$ acting by 
\begin{equation*}
\lambda(\tau)\cdot x = \tau^{-1} x,\quad \lambda(\tau)\cdot y = \tau y.
\end{equation*}
We let $\bar{\chi}$ be the limit of $\chi$ under the action of $\lambda$ as in the proof of Proposition \ref{specFILTDef}. 
\begin{prop} The limit filtration $\bar{\chi}$ is not finitely generated.
\end{prop} 
\begin{proof} The 1-parameter subgroup $\lambda$ induces a term ordering $>$ on the $\C[t]$-algebra $\C[t][x, y]$ which is compatible with the graded $\C[t]$-algebra structure and for which we have $x > y$. Let us denote the initial term of an element $\sigma \in \C[t][x, y]$ by $\init_{>} \sigma$. The Rees algebra $\rees(\bar{\chi})$ coincides with the initial algebra of $A$ defined by
\begin{equation*}
\init_{>} A = \{ \init_{>} \sigma \, : \, \sigma \in A\, \}.
\end{equation*}
We show that $\init_{>} A$ is not finitely generated. The proof follows closely the original argument in \cite{robbiano} Example 1.20.  

\emph{Claim 1. The algebra $A$ contains all the monomials of the form $t^{n-1} x y^n$ for $n \geq 3$, and does not contain elements which have a homogeneous component of the form $t^k x y^n$ for $k < n - 1$. In particular no element of $A$ can have initial term of the form $t^k x y^n$ for $k < n - 1$.}  To check the first statement we observe that we have for $n \geq 3$
\begin{equation*}
t^{n - 1} x y^n = t (x + y) t^{n - 2} x y^{n - 1} - t (x y) t (t^{n - 3} x y^{n - 2})
\end{equation*}
and then argue by induction starting from the fact that $A$ contains the monomials $t (x + y), t x y, t x y^2$. For the second statement it is enough to check that $A$ does not contain $t^k x y^n$ for $k < n - 1$ (since $A$ is a graded subalgebra). This is a simple check.

\emph{Claim 2. The algebra $A$ does not contain elements which have a homogeneous component of the form $t^{k} y^j$ for $k \leq j$. In particular no element of $A$ can have initial term of the form $t^{k} y^j$ for $k \leq j$}. Since $A$ is a graded subalgebra it is enough to show that $t^{k} y^j$ cannot belong to $A$ if $k \leq j$. All the elements of $A$ are of the form $f(t(x+ y), t xy, t xy^2, t^2 y)$ where $f(x_1, x_2, x_3, x_4)$ is a polynomial with coefficients in $\C[t]$. Assuming 
\begin{equation*}
f(t(x+ y), t xy, t xy^2, t^2 y) = t^k y^j
\end{equation*}  
and setting $y = 0$ gives $f(t x, 0, 0, 0) = 0$. Similarly setting $x = 0$ gives $f(t y, 0, 0, t^2 y)= t^k y^j$. If $k \leq j$ it follows that necessarily $k = j$ and $f(x_1, 0, 0, x_2) = x_1$. Comparing with $f(t x, 0, 0, 0)$ we find $t x = 0$, a contradiction. 

\emph{Claim 3. $\init_{>} A$ is not finitely generated.} Assuming $\init_{>} A$ is finitely generated we can find a finite set $\sigma_i$ of elements of $A$ such that $\init_> \sigma_i$ generate $\init_> A$. By finiteness we can choose $m \gg 1$ such that for all $i$ we have $\init_> \sigma_i \neq t^{m-1} x y^m$. On the other hand by Claim 1 we know that for all $m$ we have $t^{m-1} x y^m \in \init_> A$. By the definition of a term ordering we know thus that $t^{m-1} x y^m$ must be a product of powers of initial terms of the elements $\sigma_i$. As $x$ appears linearly it follows that there must be two generators $\sigma_i$, $\sigma_j$ with $\init_{>} \sigma_i = t^p x y^{r}$, respectively $\init_{>} \sigma_j = t^q y^s$ with $p + q = m - 1$, $r + s = m$. By Claim 1 we must have $p \geq r - 1$ and by Claim 2 we must have $q > s$. Hence $p + q > r + s - 1 = m - 1$ so $p + q \geq m$, a contradiction.

\end{proof}

\end{document}